\documentclass[leqno,12pt]{amsart}
\usepackage{amsfonts}
\usepackage{amsmath,amssymb,amsthm}

\everymath{\displaystyle}

\newtheorem{thm}{Theorem}[section]
\newtheorem{cor}[thm]{Corollary}

\theoremstyle{definition}
\newtheorem{definition}[thm]{Definition}

\theoremstyle{remark}
\newtheorem{rem}[thm]{Remark}

\numberwithin{claim}{thm}
\numberwithin{equation}{section}

\begin{document}

\title[Natural mates of non-null Frenet curves in the Minkowski 3-space]
{Natural mates of non-null Frenet curves in the Minkowski 3-space}

\author{Alev Kelleci}
\address{F\i rat University, Faculty of Science, Department of Mathematics, 23200 Elaz\i\u g, Turkey.}
\email{alevkelleci@hotmail.com}


\begin{abstract}
In this paper, we give the definition of the natural mate of a non-null Frenet curve in Minkowski 3-spaces. The main purpose of this paper is to prove some relationships between a non-null Frenet curve and its natural mate. In particular, we obtain some necessary and sufficient conditions for the natural mate of a non-null Frenet curve to be a slant helix, a spherical curve, or a curve of constant curvature. Several applications of our main results are also presented. 
\end{abstract}

\maketitle

\section{Introduction}\label{sec:1}

In the study of the fundamental theory and the characterizations of space curves, the corresponding 
relations between the curves are the very interesting and important problem not only in 
Euclidean space but also Lorentzian space. One of the most popular among these curves is Frenet curve. 
If a curve $\gamma$ is a Frenet curve, then its curvature $\kappa > 0$ and torsion $\tau \neq 0$. 
Some important kinds of these curves are slant helices, spherical curves, rectifying curves, 
characterized in []. Since there exist three kinds of curves (time-like, space-like, and null 
or light-like curves) depending on their causal characters in a Lorentzian space, working in 
Lorentzian space is more complicated than working in Euclidean space. Also, it is well-known that
 the studies of space-like curves and time-like curves have many analogies and similarities because
 they have the natural geometric invariant parameter by the arc-length parameter which normalizes 
the tangent vector, \cite{ONeillKitap}. 

Recently, the theory of the associated curve of a given curve has been one of interesting topics.
 Many geometers have investigated this problem from different viewpoints: For example, the Bertrand 
partner and the Mannheim partner curves in Minkowski 3-space, are two important types of associated
 curves, characterized by the curvature and the torsion, (see in \cite{LW2008,BBE2004,CKK2012,CKK2013}).
 In \cite{CKA2012}, the authors studied also the general helices and the slant helices in 
Minkowski 3-spaces, by using some special associated curves of a given curve. Subsequently, they
studied the Euclidean version in \cite{CK2012}. They called the special associated curve
the principal-directional (-donor) curve and the binormal-directional (-donor)
curve. These notions gave us a certain method constructing the general helices
and the slant helices (see \cite{CK2012,CKA2012}). In \cite{DCA2018}, Deshmukh, Chen and Alghanemi 
studied a new type associated curve called as the natural mate of a Frenet curve in Euclidean 3-space, 
closely related with the principal-directional (-donor) curve defined in \cite{CK2012,CKA2012}. 
They characterized these curves and also gave new results for them. In this paper, we will define 
the concept of natural mate for non-null Frenet curves in Minkowski 3-space by moving 
from the notion of natural mate defined in \cite{DCA2018}. We have also shown that the natural
 mate of non-null Frenet curves is a unique and obtained some characterizations for them.  

All geometric objects under consideration are smooth and curves are regular
unless otherwise stated.

\section{Preliminaries}\label{SectPre}
In this section, we would like to give a brief summary of basic definitions, facts and equations in the theory of curves and surfaces in Minkowski $3$-space (see for detail, \cite{ONeillKitap,Spivak,Kuhnel}).

Let $\mathbb E^3_1$ denote the Minkowski $3$-space with the canonical
Lorentzian metric tensor given by
$$
\left\langle \cdot,\cdot \right\rangle= dx_1^2+dx_2^2-dx_3^2,
$$
where $\left(x_1, x_2, x_3\right)$  are rectangular coordinates of the points of $\mathbb E^3_1$.

The causality of a vector in a Minkowski space is defined as following.
A non-zero vector $v$ in $\mathbb{E}^3_1$ is said to be space-like, time-like
and light-like (null) regarding to $\left\langle v,v\right\rangle >0$ , $%
\left\langle v,v\right\rangle <0$ and $\left\langle v,v\right\rangle =0$,
respectively. We consider the zero vector as a space-like vector. 
Note that $v$ is said to be causal if it is not space-like. Two non-zero vectors 
$u$ and $v$ in $\mathbb{E}^3_1$ are said to be orthogonal if $\left\langle u,v\right\rangle=0$.
 A set of $\left\{e_1,e_2,e_3\right\}$ of vectors in $\mathbb{E}^3_1$ is called an orthonormal 
frame if it satisfies that 
$$
\left\langle e_1,e_1\right\rangle=-1, \quad \left\langle e_2,e_2\right\rangle=\left\langle e_3,e_3\right\rangle=1, %
 \quad \left\langle e_i,e_j\right\rangle=0, i\neq j.
$$
For two non-zero vectors $u=\left(u_1,u_2,u_3\right)$ and $v=\left(v_1,v_2,v_3\right)$ in 
$\mathbb E^3_1$, we define the (Lorentzian) vector product of $u$ and $v$ as following:
$$
u\times v=\left(u_3v_2-u_2v_3, u_3v_1-u_1v_3, u_1v_2-u_2v_1\right).
$$
One check that the vector product is skew-symmetric, i.e., $u\times v = -v\times u$.

A curve $\gamma=\gamma(t)$ in $\mathbb{E}^3_1$ is said to be space-like, time-like or null (light-like) if 
its tangent vector field $\gamma^{\prime}(t)$ is space-like, time-like or null (light-like), respectively, for all $t$.
 
Let $\gamma$ be a non-null curve in $\mathbb{E}^3_1$ parametrized by arc-length, i.e., 
$\left|\left\langle \gamma^{\prime},\gamma^{\prime}\right\rangle\right|=1$, and we suppose that 
$\left|\left\langle \gamma{''},\gamma{''}\right\rangle\right|\neq 0$. Then this curve $\gamma$ 
induces a Frenet frame \linebreak $\displaystyle \left\{T=\gamma^{\prime}, N=\frac{\gamma''}{\sqrt{\left|\left\langle \gamma{''},\gamma{''}\right\rangle\right|}}, B=T\times N\right\}$ satisfying the following Frenet equations:

\begin{align} \label{Ffnn}
\begin{split}
\left[ \begin{array}{c}
T'\\
N'\\
B'
\end{array}\right]%
=%
\left[ \begin{array}{ccc}
0&\kappa \varepsilon_1&0\\
-\kappa \varepsilon_0&0&-\tau \varepsilon_0 \varepsilon_1\\
0&-\tau \varepsilon_1&0
\end{array}\right]\left[ \begin{array}{c}
T\\
N\\
B
\end{array}\right]
\end{split}
\end{align}
where $\left\langle T,T\right\rangle=\varepsilon_0, \left\langle N,N\right\rangle=\varepsilon_1,%
 \left\langle B,B\right\rangle=-\varepsilon_0 \varepsilon_1, \left\langle T',N\right\rangle=\kappa$ and $\left\langle N',B\right\rangle=\tau.$ The vector fields 
$T, N, B$ and the functions $\kappa, \tau$ are called the tangent, principal normal, binormal and curvature 
and torsion of $\gamma$, respectively. Accordingly, the Frenet frame of $\gamma$ satisfies 
$$
T \times N= B, \quad N \times B= -\varepsilon_1 T, \quad B \times T= -\varepsilon_0 N.
$$
In \eqref{Ffnn}, if $\varepsilon_0=1$ or 
$\varepsilon_0=-1$, then $\gamma$ is space-like or time-like, respectively. A space-like curve $\gamma$ 
satisfying \eqref{Ffnn} is said to be \textit{type1} or \textit{type2} if $\varepsilon_1=1$ or 
$\varepsilon_1=-1$, respectively. 

When the Frenet frame moves along a curve in $\mathbb E^3_1$, there exist an axis of instantaneaus frame's 
rotation (Darboux). The direction of such axis is given by Darboux vector. If $\gamma$ is a unit speed non-null 
curve with the Frenet-Serret apparatus $\left\{\kappa_{\gamma},\tau_{\gamma},T_{\gamma},N_{\gamma},B_{\gamma}\right\}$, the Darboux vector of $\gamma$ is that 
\begin{equation} \label{Dvgnn}
D_{\gamma}(s)=-\varepsilon_0 \varepsilon_1 \tau_{\gamma}(s)T_{\gamma}(s)-\varepsilon_0 \varepsilon_1 \kappa_{\gamma}(s)B_{\gamma}(s),
\end{equation}
whose length is given by
\begin{equation} \label{wnn}
\omega_{\gamma}=\sqrt{\left|\varepsilon_0 {\tau_{\gamma}}^2-\varepsilon_0 \varepsilon_1 {\kappa_{\gamma}}^2\right|}.
\end{equation}
Note that the Darboux vector defined in \eqref{Dvgnn} is also called as the centrode of $\gamma$ in $\mathbb{E}^3_1$. Thus, the Darboux equations are defined as
\begin{subequations}
\begin{eqnarray}
T^{\prime}_{\gamma}=D_{\gamma}\times T_{\gamma},\\
N^{\prime}_{\gamma}=D_{\gamma}\times N_{\gamma},\\
B^{\prime}_{\gamma}=D_{\gamma}\times B_{\gamma}.
\end{eqnarray}
\end{subequations}
The co-centrode of $\gamma$ from the point of \cite{DCA2018} is given by 
\begin{equation} \label{cDvgnn}
D^{*}_{\gamma}(s)=-\varepsilon_0 \varepsilon_1 \tau_{\gamma}(s)B_{\gamma}(s)-\varepsilon_0 \kappa_{\gamma}(s)T_{\gamma}(s),
\end{equation}
which is exactly the derivative of the principal normal vector of $\gamma$ defined in \eqref{Ffnn}. 

Recently, in \cite{DCA2018} authors have introduced the concept of the natural mate of a curve 
in $\mathbb E^3$. Motivated by what happens in Euclidean ambient space, we would like to give next definition:

\begin{definition}
Let $\gamma:I \rightarrow \mathbb E^3_1$ be a non-null Frenet curve in $\mathbb E^3_1$ parametrized 
by the arc-length parameter with Frenet frame $\left\{T_{\gamma},N_{\gamma},B_{\gamma}\right\}$. A curve 
$\beta:I \rightarrow \mathbb E^3_1$ is called the natural mate of the curve $\gamma$, if the tangent $T_{\beta}$ 
is equal to $N_{\gamma}$, i.e., $T_{\beta}=N_{\gamma}$. 
\end{definition}


\section{Some known results} \label{S:Results}

In this section we would like to give some theorems obtained before, which will be play an 
important role in the proofs of our results. 

\begin{thm} \label{Rcnn}
\cite{INP2003} Let $\gamma=\gamma(s)$ be a unit speed non-null curve in $\mathbb E^3_1$, with a space-like or 
a time-like rectifying plane and with the curvature $\kappa(s)>0$. Then up to the isometries 
of $\mathbb E^3_1$, the curve $\gamma$ is rectifying if and only if there holds 
\begin{equation} \label{nnrccondition}
\tau(s)/ \kappa(s)= c_1 s+c_2,
\end{equation}
 where $c_1\in R_0, c_2\in R$.
\end{thm}

\begin{thm} \label{SShteo}
\cite{AL2011} Let $\gamma$ be a unit speed space-like curve in $\mathbb E^3_1$.
\begin{enumerate}
\item If the normal vector of $\gamma$ is space-like, then $\gamma$ is a slant helix if and 
only if either one the next two functions 
\begin{equation} \label{Shtype1}
\sigma_{s1}=\frac{\kappa^2}{(\kappa^2-\tau^2)^{3/2}}{\left(\frac{\tau}{\kappa}\right)}^{\prime}\quad or\quad %
\sigma_{s2}=\frac{\kappa^2}{(\tau^2-\kappa^2)^{3/2}}{\left(\frac{\tau}{\kappa}\right)}^{\prime}
\end{equation}
is constant everywhere $\tau^2-\kappa^2$ does not vanish.
\item If the normal vector of $\gamma$ is time-like, then $\gamma$ is a slant helix if and 
only if the functions
\begin{equation} \label{Shtype2}
\sigma_{s3}=\frac{\kappa^2}{(\tau^2+\kappa^2)^{3/2}}{\left(\frac{\tau}{\kappa}\right)}^{\prime}
\end{equation}
is constant.
\end{enumerate}
\end{thm}

\begin{thm} \label{TShteo}
\cite{AL2011} Let $\gamma$ be a unit speed time-like curve in $\mathbb E^3_1$.
Then $\gamma$ is a slant helix if and only if either one the next two functions 
\begin{equation} \label{TSh}
\sigma_{t1}=\frac{\kappa^2}{(\tau^2-\kappa^2)^{3/2}}{\left(\frac{\tau}{\kappa}\right)}^{\prime}\quad or\quad %
\sigma_{t2}=\frac{\kappa^2}{(\kappa^2-\tau^2)^{3/2}}{\left(\frac{\tau}{\kappa}\right)}^{\prime}
\end{equation}
is constant everywhere $\tau^2-\kappa^2$ does not vanish.
\end{thm}

\begin{thm} \label{SScteo}
\cite{PP1999,PS2000} Let $\gamma$ be a unit speed space-like curve in $\mathbb E^3_1$.
\begin{enumerate}
\item If the normal vector of $\gamma$ is space-like, then $\gamma$ is a spherical curve if and 
only if 
\begin{equation} \label{Sctype1}
\frac{\tau}{\kappa}=\left[\frac{1}{\tau}\left(\frac{1}{\kappa}\right)^{\prime}\right]^{\prime}\quad and\quad %
\left(\frac{1}{\kappa}\right)^{2}>\left[\frac{1}{\tau}\left(\frac{1}{\kappa}\right)^{\prime}\right]^{2}
\end{equation}
holds, where $\frac{1}{\tau},\frac{1}{\kappa}$ does not vanish.
\item If the normal vector of $\gamma$ is time-like, then $\gamma$ is a spherical curve if and 
only if 
\begin{equation} \label{Sctype2}
\frac{\tau}{\kappa}=\left[\frac{1}{\tau}\left(\frac{1}{\kappa}\right)^{\prime}\right]^{\prime}\quad and\quad %
\left(\frac{1}{\kappa}\right)^{2}<\left[\frac{1}{\tau}\left(\frac{1}{\kappa}\right)^{\prime}\right]^{2}
\end{equation}
holds, where $\frac{1}{\tau},\frac{1}{\kappa}$ does not vanish.
\end{enumerate}
\end{thm}

\begin{thm} \label{TScteo}
\cite{PS2001} Let $\gamma$ be a unit speed time-like curve in $\mathbb E^3_1$.
Then $\gamma$ is a spherical curve if and only if 
\begin{equation} \label{TSc}
\frac{\tau}{\kappa}=-\left[\frac{1}{\tau}\left(\frac{1}{\kappa}\right)^{\prime}\right]^{\prime}
\end{equation}
holds, where $\tau,\kappa$ does not vanish.
\end{thm}

\section{Natural mates of non-null Frenet curves}

In this section, first we recall the following results from \cite{DCA2018}. 

\begin{thm} \label{TeoS}
Let $\gamma$ be a space-like Frenet curve in $\mathbb E^3_1$ with Frenet-Serret apparatus $\left\{\kappa_{\gamma},\tau_{\gamma},T_{\gamma},N_{\gamma},B_{\gamma}\right\}$. Then, there exist a unit speed non-null curve $\beta$ with Frenet-Serret apparatus $\left\{\kappa_{\beta},\tau_{\beta},T_{\beta},N_{\beta},B_{\beta}\right\}$ are given by 
\begin{itemize}
\item[(i)] 
\begin{subequations} \label{Faofb1ALL}
\begin{eqnarray} \label{Faofb1a}
\kappa_{\beta}=\sqrt{{\kappa_{\gamma}}^2-{\tau_{\gamma}}^2},\quad \tau_{\beta}=\frac{{\kappa_{\gamma}}^2}{{\kappa_{\gamma}}^2-{\tau_{\gamma}}^2}{\left(\frac{\tau_{\gamma}}{\kappa_{\gamma}}\right)^{\prime}},\\ \label{Faofb1b}
T_{\beta}=N_{\gamma},\quad N_{\beta}=\frac{D^{*}_{\gamma}}{\sqrt{{\kappa_{\gamma}}^2-{\tau_{\gamma}}^2}},\quad %
B_{\beta}= \frac{D_{\gamma}}{\sqrt{{\kappa_{\gamma}}^2-{\tau_{\gamma}}^2}}
\end{eqnarray}
\end{subequations} 
such that $\left|\kappa\right|>\left|\tau\right|$. In case, $\beta$ is a space-like curve of \textit{type1} in $\mathbb E^3_1$.
\item[(ii)]
\begin{subequations} \label{Faofb2ALL}
\begin{eqnarray} \label{Faofb2a}
\kappa_{\beta}=\sqrt{{\tau_{\gamma}}^2-{\kappa_{\gamma}}^2},\quad \tau_{\beta}=\frac{{\kappa_{\gamma}}^2}{{\tau_{\gamma}}^2-{\kappa_{\gamma}}^2}{\left(\frac{\tau_{\gamma}}{\kappa_{\gamma}}\right)^{\prime}},\\ \label{Faofb2b}
T_{\beta}=N_{\gamma},\quad N_{\beta}=\frac{D^{*}_{\gamma}}{\sqrt{{\tau_{\gamma}}^2-{\kappa_{\gamma}}^2}},\quad %
B_{\beta}= \frac{D_{\gamma}}{\sqrt{{\tau_{\gamma}}^2-{\kappa_{\gamma}}^2}}
\end{eqnarray}
\end{subequations} 
such that $\left|\kappa\right|<\left|\tau\right|$. In case, $\beta$ is a space-like curve of \textit{type2} in $\mathbb E^3_1$.
\item[(iii)]
\begin{subequations} \label{Faofb3ALL}
\begin{eqnarray} \label{Faofb3a}
\kappa_{\beta}=\sqrt{{\kappa_{\gamma}}^2+{\tau_{\gamma}}^2},\quad \tau_{\beta}=\frac{{\kappa_{\gamma}}^2}{{\kappa_{\gamma}}^2+{\tau_{\gamma}}^2}{\left(\frac{\tau_{\gamma}}{\kappa_{\gamma}}\right)^{\prime}},\\ \label{Faofb3b}
T_{\beta}=N_{\gamma},\quad N_{\beta}=\frac{D^{*}_{\gamma}}{\sqrt{{\tau_{\gamma}}^2+{\kappa_{\gamma}}^2}},\quad %
B_{\beta}= \frac{D_{\gamma}}{\sqrt{{\tau_{\gamma}}^2+{\kappa_{\gamma}}^2}}
\end{eqnarray}
\end{subequations}
In case, $\beta$ is a time-like curve in $\mathbb E^3_1$.
\end{itemize}
\end{thm}

\begin{proof}
Assume that $\gamma$ be a space-like Frenet curve in $\mathbb E^3_1$. Since $\gamma$ can be of \textit{type1} 
or \textit{type2}, we consider on these situations. However, as the proof of these situations will be very 
similar to each other just for the second case will be enough to prove. Thus, let $\gamma$ be a space-like
 Frenet curve of \textit{type2}. Obviously, the set of $\left\{N_{\gamma},\frac{D^{*}_{\gamma}}{\omega_{\gamma}},\frac{D_{\gamma}}{\omega_{\gamma}}\right\}$ is an orthonormal frame along $\gamma$ of \textit{type2} in $\mathbb E^3_1$, where $D_{\gamma},\omega_{\gamma},D^{*}_{\gamma}$ are defined in \eqref{Dvgnn}, \eqref{wnn}, \eqref{cDvgnn}, respectively, with $\varepsilon_0=1,\varepsilon_1=-1$. 
By considering \eqref{Shtype2} and using, 
\begin{equation} \nonumber
\left(\frac{\kappa_{\gamma}}{\omega_{\gamma}}\right)^{\prime}=-\tau_{\gamma}\sigma_{s3}\quad and \quad%
\left(\frac{\tau_{\gamma}}{\omega_{\gamma}}\right)^{\prime}=\kappa_{\gamma}\sigma_{s3}\quad 
\end{equation}
and the Frenet formulas given in \eqref{Ffnn} with $\varepsilon_0=1,\varepsilon_1=-1$, we get
\begin{equation} \label{FfTbeta}
{N_{\gamma}}^{\prime}=D^{*}_{\gamma},\quad \left(\frac{D^{*}_{\gamma}}{\omega_{\gamma}}\right)^{\prime}=\omega_{\gamma}N_{\gamma}+\sigma_{s3}\omega_{\gamma}\left(\frac{D_{\gamma}}{\omega_{\gamma}}\right),\quad\left(\frac{D_{\gamma}}{\omega_{\gamma}}\right)^{\prime}=-\sigma_{s3}\omega_{\gamma}\left(\frac{D^{*}_{\gamma}}{\omega_{\gamma}}\right).
\end{equation}
By virtue of existence theorem, the equations obtained in \eqref{FfTbeta} make assure that there exists a 
unit speed time-like curve $\beta$ in $\mathbb E^3_1$ with Frenet-Serret apparatus $\left\{\kappa_{\beta},\tau_{\beta},T_{\beta},N_{\beta},B_{\beta}\right\}$ described as in (iii) in the theorem.
As previously mentioned, the proofs of (i) and (ii) in theorem can be made by taking $\varepsilon_0 =1,\varepsilon_1=1$ for the curve $\gamma$ exactly the same as the path we follow for proof of (iii).
\end{proof}

\begin{thm} \label{TeoT}
Let $\gamma$ be a time-like Frenet curve in $\mathbb E^3_1$ with Frenet-Serret apparatus \linebreak $\left\{\kappa_{\gamma},\tau_{\gamma},T_{\gamma},N_{\gamma},B_{\gamma}\right\}$. Then, there exist a unit speed space-like curve $\beta$ with Frenet-Serret apparatus $\left\{\kappa_{\beta},\tau_{\beta},T_{\beta},N_{\beta},B_{\beta}\right\}$ are given by 
\begin{itemize}
\item[(i)]
\begin{subequations} \label{Faofb4ALL}
\begin{eqnarray} \label{Faofb4a}
\kappa_{\beta}=\sqrt{{\tau_{\gamma}}^2-{\kappa_{\gamma}}^2}\quad \tau_{\beta}=\frac{{\kappa_{\gamma}}^2}{{\tau_{\gamma}}^2-{\kappa_{\gamma}}^2}{\left(\frac{\tau_{\gamma}}{\kappa_{\gamma}}\right)^{\prime}},\\ \label{Faofb4b}
T_{\beta}=N_{\gamma},\quad N_{\beta}= \frac{D^{*}_{\gamma}}{\sqrt{{\tau_{\gamma}}^2-{\kappa_{\gamma}}^2}},\quad %
B_{\beta}=\frac{D_{\gamma}}{\sqrt{{\tau_{\gamma}}^2-{\kappa_{\gamma}}^2}},
\end{eqnarray}
\end{subequations} 
such that $\left|\kappa\right|<\left|\tau\right|$. In case, $\beta$ is a space-like curve of \textit{type1} in $\mathbb E^3_1$.
\item[(ii)]
\begin{subequations} \label{Faofb5ALL}
\begin{eqnarray} \label{Faofb5a}
\kappa_{\beta}=\sqrt{{\kappa_{\gamma}}^2-{\tau_{\gamma}}^2}\quad \tau_{\beta}=\frac{{\kappa_{\gamma}}^2}{{\kappa_{\gamma}}^2-{\tau_{\gamma}}^2}{\left(\frac{\tau_{\gamma}}{\kappa_{\gamma}}\right)^{\prime}},\\ \label{Faofb5b}
T_{\beta}=N_{\gamma},\quad N_{\beta}= \frac{D^{*}_{\gamma}}{\sqrt{{\kappa_{\gamma}}^2-{\tau_{\gamma}}^2}},\quad %
B_{\beta}=\frac{D_{\gamma}}{\sqrt{{\kappa_{\gamma}}^2-{\tau_{\gamma}}^2}}
\end{eqnarray}
\end{subequations} 
such that $\left|\kappa\right|>\left|\tau\right|$. In case, $\beta$ is a space-like curve of \textit{type2} in $\mathbb E^3_1$.
\end{itemize}
Here, $D_{\gamma}$ and $D^{*}_{\gamma}$ are defined by \eqref{Dvgnn} and \eqref{cDvgnn}, respectively.
\end{thm}

\begin{rem}
We would like to note that in \cite{DCA2018}, the authors explained why they chose to name the pair of 
$\left\{\gamma,\beta\right\}$ as the natural mates in $\mathbb E^3$. For the same reason, we will call the curve 
$\beta$ as the natural mate of $\gamma$ in $\mathbb E^3_1$ instead of the principal-direction curve defined in \cite{CKA2012}.
\end{rem}

The followings are some consequences obtained easily from Theorem \ref{TeoS} and Theorem \ref{TeoT}.

\begin{cor}
[\cite{CKA2012}, Lemma 4.1] A non-null Frenet curve $\gamma$ in $\mathbb E^3_1$ is a general helix 
with $\tau_{\gamma} \neq \kappa_{\gamma}$ if and only if its natural mate $\beta$ is a plane curve.
\end{cor}

\begin{cor}
[\cite{CKA2012}, Corollary 4.7] A space-like Frenet curve $\gamma$ of \textit{type1} in $\mathbb E^3_1$ 
is a helix with $\left|\kappa\right|>\left|\tau\right|$ or $\left|\kappa\right|<\left|\tau\right|$ 
if and only if its space-like natural mate $\beta$ is a circle in $\mathbb E^2$ or a space-like hyperbola in 
$\mathbb E^2_1$, respectively.
\end{cor}

\begin{cor}
[\cite{CKA2012}, Corollary 4.9] A time-like Frenet curve $\gamma$ in $\mathbb E^3_1$ 
is a helix with $\left|\kappa\right|<\left|\tau\right|$ or $\left|\kappa\right|>\left|\tau\right|$ 
if and only if its space-like natural mate $\beta$ is a circle in $\mathbb E^2$ or a space-like hyperbola in 
$\mathbb E^2_1$, respectively.
\end{cor}

\begin{cor}
[\cite{CKA2012}, Proposition 5.3] A non-null Frenet curve $\gamma$ in $\mathbb E^3_1$ is a slant helix if and only if its natural mate 
is a general helix with $\frac{\tau_{\beta}}{\kappa_{\beta}}=\sigma_{\gamma}$.
\end{cor}

Now, we also have the following characterizations in term of natural mates by applying Theorem \ref{TeoS} and Theorem \ref{TeoT}.

\begin{cor}
A non-null curve $\gamma$ in $\mathbb E^3_1$ with curvature $\kappa_{\gamma}$ and 
torsion $\tau_{\gamma}$ is a spherical curve lying on a sphere on radius $r$ if and only if the curvature
 $\kappa_{\beta}$ and torsion $\tau_{\beta}$ of its natural mate $\beta$ satisfy one of the followings
\begin{enumerate}
\item
\begin{equation}
\frac{\kappa^{\prime}_{\beta}}{\kappa_{\beta}}=-\tau_{\beta}\left(\frac{\tau_{\gamma}}{\kappa_{\gamma}}\right)\pm \tau_{\gamma}%
\sqrt{1-r^2 {\kappa_{\gamma}}^2},
\end{equation}
where $\kappa_{\beta},\tau_{\beta}$ is defined as in \eqref{Faofb1a}.
\item
\begin{equation}
\frac{\kappa^{\prime}_{\beta}}{\kappa_{\beta}}=\tau_{\beta}\left(\frac{\tau_{\gamma}}{\kappa_{\gamma}}\right)\pm \tau_{\gamma}%
\sqrt{1+r^2 {\kappa_{\gamma}}^2},
\end{equation}
where $\kappa_{\beta},\tau_{\beta}$ is defined as in \eqref{Faofb2a}.
\item
\begin{equation}
\frac{\kappa^{\prime}_{\beta}}{\kappa_{\beta}}=\tau_{\beta}\left(\frac{\tau_{\gamma}}{\kappa_{\gamma}}\right)\pm \tau_{\gamma}%
\sqrt{r^2 {\kappa_{\gamma}}^2-1},
\end{equation}
where $\kappa_{\beta},\tau_{\beta}$ is defined as in \eqref{Faofb3a}.
\item
\begin{equation}
\frac{\kappa^{\prime}_{\beta}}{\kappa_{\beta}}=\tau_{\beta}\left(\frac{\tau_{\gamma}}{\kappa_{\gamma}}\right)\pm \tau_{\gamma}%
\sqrt{r^2 {\kappa_{\gamma}}^2-1},
\end{equation}
where $\kappa_{\beta},\tau_{\beta}$ is defined as in \eqref{Faofb4a}.
\item
\begin{equation}
\frac{\kappa^{\prime}_{\beta}}{\kappa_{\beta}}=-\tau_{\beta}\left(\frac{\tau_{\gamma}}{\kappa_{\gamma}}\right)\pm \tau_{\gamma}%
\sqrt{1-r^2 {\kappa_{\gamma}}^2},
\end{equation}
where $\kappa_{\beta},\tau_{\beta}$ is defined as in \eqref{Faofb5a}.
\end{enumerate}
\end{cor}

\begin{proof}
The above formulas are proved exactly the same way as in the Euclidean case such as,
 for instance, the elementary proof given in \cite{DCA2018}.
\end{proof}

\begin{cor}
A non-null Frenet frame $\gamma$ with curvature $\kappa_{\gamma}$ and torsion $\tau_{\gamma}$ 
is a rectifying curve in $\mathbb E^3_1$ if and only if the curvature $\kappa_{\beta}$ and 
torsion $\tau_{\beta}$ of its natural mate $\beta$ satisfy
\begin{equation} \label{cnn}
\tau_{\beta}=c\left(\frac{\kappa_{\gamma}}{\kappa_{\beta}}\right)^2
\end{equation}
where $c$ is non-zero constant.
\end{cor}

\begin{proof}
Let $\gamma$ be a non-null rectifying curve with curvature $\kappa_{\gamma}$ and torsion 
$\tau_{\gamma}$. Then from \eqref{nnrccondition}, we can write the curve $\gamma$ as
\begin{equation}\nonumber
\gamma(s)=(s+c_1)T_{\gamma}+c_2 B_{\gamma},
\end{equation} 
for some constants $c_1$ and $c_2 \neq 0$. By differentiating the last above equation, we get 
$c_2 \tau_{\gamma}=(s+c_1)\kappa_{\gamma}$ which yields
\begin{equation}\nonumber
\left(\frac{\tau_{\gamma}}{\kappa_{\gamma}}\right)^{\prime}=c
\end{equation}
where $c=\frac{1}{c_2}$. Now, we get directly \eqref{cnn} by using expressions of $\kappa_{\beta}$
 and $\tau_{\beta}$. 
Conversely, let the condition \eqref{cnn} be hold. By substituting the expressions of $\kappa_{\beta}$
 and $\tau_{\beta}$ in \eqref{cnn}, we get
\begin{equation}\nonumber
\left(\frac{\tau_{\gamma}}{\kappa_{\gamma}}\right)^{\prime}=c
\end{equation}
for $c=\frac{1}{c_2}$. So we conclude ${\tau_{\gamma}}/{\kappa_{\gamma}}$ is a linear function 
of $s$, which say that $\gamma$ is a rectifying curve according to \eqref{nnrccondition}.  
\end{proof}

\section{Some relationship with Natural mates}

\begin{thm} \label{Sthmsc}
Let $\gamma$ be a space-like Frenet curve with constant curvature $\kappa_{\gamma}=c>0$. Then its natural 
mate $\beta$ is a non-null spherical curve with radius $r$. The opposite is provided if $\tau_{\beta}$, 
the torsion of $\beta$ is a non-zero.
\end{thm}

\begin{proof}
Let $\gamma$ be a space-like Frenet curve with $\kappa_{\gamma}=c>0$. Since $\gamma$ can be of \textit{type1} 
or \textit{type2}, we consider on these situations. However, as the proof of these situations will be very 
similar to each other just for the first case will be enough to prove. Thus, let $\gamma$ be a space-like
 Frenet curve of \textit{type1}. Note that since $\beta$ is a space-like curve of \textit{type1} 
or \textit{type2} for the first assumption from Theorem \ref{TeoS}, we assume 
the natural mate $\beta$ is a space-like curve of \textit{type1}. Thus, the curvature $\kappa_{\beta}$ 
and the torsion $\tau_{\beta}$ of the space-like natural mate $\beta$ of $\gamma$ given in \eqref{Faofb1a}
 can be rewritten as
\begin{equation} \label{ctofssb}
\kappa_{\beta}=\sqrt{{c}^2-{\tau_{\gamma}}^2},\quad \tau_{\beta}=\frac{c\left({\tau}{\kappa}\right)^{\prime}}{{c}^2-{\tau_{\gamma}}^2}.
\end{equation}

Now, if the space-like curve $\gamma$ of \textit{type1} has also a constant torsion, i.e., 
$\tau_{\gamma}=c_1$, then we directly conclude that  $\kappa_{\beta}=c_2$ and $\tau_{\beta}=0$ of the curve $\beta$, from the above expression. Which says that $\beta$ is a part of circle with radius $c^{-1}$, so we get the natural mate $\beta$ of \textit{type1} is a spherical curve lying on a sphere of radius $c^{-1}$.

On the other hand, let the space-like curve $\gamma$ of \textit{type1} be of a non-constant torsion. By taking into account the expressions 
\eqref{ctofssb}, we can show that the following equation is provided by direct calculation 
\begin{equation} \label{SScssbc}
\frac{\tau_{\beta}}{\kappa_{\beta}}=\left[\frac{1}{\tau_{\beta}}\left(\frac{1}{\kappa_{\beta}}\right)^{\prime}\right]^{\prime},
\end{equation} 
saying that the space-like natural mate $\beta$ of \textit{type1} is a spherical curve from \eqref{Sctype1}. Also, we get 
\begin{equation} \label{SScssbcr}
\left(\frac{1}{\kappa_{\beta}}\right)^{2}-\left[\frac{1}{\tau_{\beta}}\left(\frac{1}{\kappa_{\beta}}\right)^{\prime}\right]^{2}=\frac{1}{c^2},
\end{equation} 
which implies that $\beta$ lies on the sphere of radius $c^{-1}$, (see in \cite{PP1999}).

Oppositely, assume that the natural mate $\beta$ of $\gamma$ is a spherical time-like curve with 
non-zero torsion, $\left(\tau_{\beta} \neq 0\right)$, lying on the sphere of radius $c^{-1}$. So, 
the equation \eqref{SScssbcr} holds. From there, we get directly
\begin{equation} \label{Kb}
\frac{{\kappa_{\beta}}^{\prime}}{\left|{\kappa_{\beta}}\right|\sqrt{c^2-{\kappa_{\beta}}^{2}}}= \frac{\tau_{\beta}}{c}.
\end{equation} 
By integrating the last equation, we get
\begin{equation}
\kappa_{\beta}= \frac{c}{\cosh \left(\int{\tau_{\beta}(s)ds}\right)}.
\end{equation}
On the other hand, from \eqref{ctofssb} we have 
\begin{equation} \nonumber
\tau_{\beta}=\frac{\left(\frac{\tau_{\beta}}{\kappa_{\beta}}\right)^{\prime}}{1-\left(\frac{\tau_{\beta}}{\kappa_{\beta}}\right)^{2}}.
\end{equation}
If we substitute $\tau_{\beta}=\kappa_{\beta} \tanh(\theta+c_1)$ in the above expression, 
we obtain $\tau_{\beta}=\theta^{\prime}$. Therefore, by considering the last results in \eqref{Kb} we get 
\begin{equation} \label{Kb0}
\kappa_{\beta}=\frac{c}{\cosh \left(\theta+c_1 \right)}
\end{equation}
with a suitable choice of constant.

On the other hand, by considering $\tau_{\beta}=\kappa_{\beta} \tanh(\theta+c_1)$ and 
$\kappa_{\beta}=\sqrt{{\kappa_{\beta}}^{2}-{\tau_{\beta}}^{2}}$ yields
\begin{equation} \label{Kb1}
\kappa_{\beta}=\frac{\kappa_{\gamma}}{ \cosh \left(\theta+c_1 \right)}
\end{equation}
Finally, by comparing the expressions $\kappa_{\beta}$ in the equations \eqref{Kb0} and \eqref{Kb1} concludes 
the curvature $\kappa_{\gamma}=c$. So, the proof of the theorem is completed.
\end{proof}

\begin{thm}
Let $\gamma$ be a time-like Frenet curve with constant curvature $\kappa_{\gamma}=c>0$. Then its natural 
mate $\beta$ is a non-null spherical curve with radius $r$. The opposite is provided if $\tau_{\beta}$, 
the torsion of $\beta$ is a non-zero.
\end{thm}

\begin{proof}
Proof of this theorem can be made exactly similar to the proof of the previous theorem.
\end{proof}

\begin{thm} \label{Sratio}
Let $\gamma$ be a space-like Frenet curve and its natural mate $\beta$ be a non-null one in $\mathbb E^3_1$. If the natural 
mate $\beta$ has constant curvature $\kappa_{\beta}=c>0$, then the ratio of the curvatures of $\gamma$ is given by 
one of the followings:
\begin{itemize}
\item[(i)] If $\left|\kappa_{\gamma}\right|>\left|\tau_{\gamma}\right|$, 
\begin{equation} \label{casea}
\frac{\tau_{\gamma}}{\kappa_{\gamma}}= \tanh \left(\int{\tau_{\beta}(s)ds}\right)
\end{equation}
In case, $\beta$ is a space-like curve of \textit{type1}.
\item[(ii)] If $\left|\kappa_{\gamma}\right|<\left|\tau_{\gamma}\right|$,
\begin{equation}
\frac{\tau_{\gamma}}{\kappa_{\gamma}}= \coth \left(\int{\tau_{\beta}(s)ds}\right)
\end{equation}
In case, $\beta$ is a space-like curve of \textit{type2}.
\item[(iii)] 
\begin{equation}
\frac{\tau_{\gamma}}{\kappa_{\gamma}}= \tan \left(\int{\tau_{\beta}(s)ds}\right)
\end{equation}
In case, $\beta$ is a time-like curve in $\mathbb E^3_1$.
\end{itemize}
\end{thm}

\begin{proof}
The proof of this theorem will be proved only for the natural mates $\left\{\gamma,\beta \right\}$ being 
a space-like pair of \textit{type1}. The other cases can be proved by the same way. 
Now, let the curvature of the space-like natural mate $\beta$ of $\gamma$, $\kappa_{\beta}$ be a constant 
$c>0$. By considering this assumption in the first expression in \eqref{Faofb1a}, we get
\begin{equation} \label{Kbc}
{c}^2={\kappa_{\gamma}}^2-{\tau_{\gamma}}^2,
\end{equation}
which yields $\kappa_{\gamma} {\kappa_{\gamma}}^{\prime}-\tau_{\gamma} {\tau_{\gamma}}^{\prime}=0$, 
where $\left|\kappa_{\gamma}\right|>\left|\tau_{\gamma}\right|$. Substituting \eqref{Kbc} in the second expression in 
\eqref{Faofb1a} and then considering the last result in it, we obtain 
\begin{equation} \label{Tbc}
\tau_{\beta}=\frac{{\tau_{\gamma}}^{\prime}{\kappa_{\gamma}}^2-{\kappa_{\gamma}}^{\prime}\tau_{\gamma}\kappa_{\gamma}}{\kappa_{\gamma}c^2}%
=\frac{{\tau_{\gamma}}^{\prime}}{\kappa_{\gamma}}=\frac{{\tau_{\gamma}}^{\prime}}{\sqrt{c^2+{\tau_{\gamma}}^2}}.
\end{equation}
Therefore, we get $\tau_{\gamma}=c \sinh \left(\int{\tau_{\beta}(s)ds}\right)$. By considering this in \eqref{Kbc} yields 
$\kappa_{\gamma}=c \cosh \left(\int{\tau_{\beta}(s)ds}\right)$. Consequently, the ratio of curvatures $\gamma$ satisfies \eqref{casea}.
\end{proof}

\begin{thm}
Let $\gamma$ be a time-like Frenet curve and its natural mate $\beta$ be a non-null one in $\mathbb E^3_1$. If the natural 
mate $\beta$ has constant curvature $\kappa_{\beta}=c>0$, then the ratio of the curvatures of $\gamma$ is given by 
one of the followings:
\begin{itemize}
\item[(i)] If $\left|\kappa_{\gamma}\right|<\left|\tau_{\gamma}\right|$,
\begin{equation}
\frac{\tau_{\gamma}}{\kappa_{\gamma}}= \coth \left(\int{\tau_{\beta}(s)ds}\right)
\end{equation}
In case, $\beta$ is a space-like curve of \textit{type1}.
\item[(ii)] If $\left|\kappa_{\gamma}\right|>\left|\tau_{\gamma}\right|$, 
\begin{equation}
\frac{\tau_{\gamma}}{\kappa_{\gamma}}= \tanh \left(\int{\tau_{\beta}(s)ds}\right)
\end{equation}
In case, $\beta$ is a space-like curve of \textit{type2}.
\end{itemize}
\end{thm}

\begin{proof}
Proof of this theorem can be made exactly similar to the proof of the previous theorem.
\end{proof}

\begin{thm}
Let $\gamma$ be a space-like Frenet curve and its natural mate $\beta$ be of constant curvature $c>0$. 
Then $\gamma$ is a spherical curve if and only if there is a positive constant $a \geq c$ such that the torsion
 $\tau_{\beta}$ of $\beta$ with respect to a suitable arc-length parametrization $s$ satisfies one of the following:
\begin{itemize}
\item[(i)] If $\left|\kappa_{\gamma}\right|>\left|\tau_{\gamma}\right|$,
\begin{equation} \label{casei}
\tau_{\beta}= \pm \frac{c^2 \sqrt{a^2+c^2}\cos (cs)}{c^2-(a^2+c^2){\sin{^2}(cs)}}.
\end{equation}
In case, $\beta$ is a space-like curve of \textit{type1}.
\item[(ii)] If $\left|\kappa_{\gamma}\right|<\left|\tau_{\gamma}\right|$,
\begin{equation}
\tau_{\beta}= \pm \frac{c^2 \sqrt{a^2+c^2}\sinh (cs)}{c^2-(a^2+c^2){\cosh{^2}(cs)}}.
\end{equation}
In case, $\beta$ is a space-like curve of \textit{type2}.
\item[(iii)] 
\begin{equation}
\tau_{\beta}= \pm \frac{c^2 \sqrt{a^2+c^2}\cosh (cs)}{c^2-(a^2+c^2){\sinh{^2}(cs)}}.
\end{equation}
In case, $\beta$ is a time-like curve in $\mathbb E^3_1$.
\end{itemize}
\end{thm}

\begin{proof}
The proof of this theorem will be proved only for the natural mates $\left\{\gamma,\beta \right\}$ being 
a space-like pair of \textit{type1}. The other cases can be proved by the same way. 
Now, let the curvature of the space-like natural mate $\beta$ of $\gamma$, $\kappa_{\beta}$ be a constant 
$c>0$ and $\tau_{\beta}$ be given in \eqref{casei}. Also, we have from Case (i) in Theorem \eqref{Sratio} 
\begin{equation} \label{condition0}
\tau_{\gamma}=c \sinh \left(\int{\tau_{\beta}(s)ds}\right),\quad \kappa_{\gamma}=c \cosh \left(\int{\tau_{\beta}(s)ds}\right), 
\end{equation} 
which gives 
\begin{subequations} \label{ScconditionALL}
\begin{eqnarray} \label{Scconditiona}
\left(\frac{1}{\kappa_{\gamma}}\right)^{\prime}=-\frac{\tau_{\beta} \tanh \left(\int{\tau_{\beta}(s)ds}\right)}{c \cosh \left(\int{\tau_{\beta}(s)ds}\right)},\\ \label{Scconditionb}
\frac{1}{\tau_{\gamma}}=\frac{1}{c \sinh \left(\int{\tau_{\beta}(s)ds}\right)}.
\end{eqnarray}
\end{subequations}
Now, as 
\begin{equation} \label{Stanh}
\tanh \left(\pm \int{\frac{c^2 \sqrt{a^2+c^2}\cos (cs)}{c^2-(a^2+c^2){\sin^2(cs)}} ds}\right)= \pm \frac{\sqrt{a^2+c^2}}{c} \sin(cs).
\end{equation}
Differentiatng the above equation and using \eqref{casei}, we get 
\begin{equation} \label{Ssech}
\cosh^{-2} \left(\int{\tau_{\beta}(s)ds}\right)=1-\left[\left(\frac{a}{c}\right)^2+1\right] \sin^2(cs).
\end{equation}
On the other hand, by using \eqref{casei} and \eqref{Ssech} in multiplication of 
\eqref{Scconditiona} and \eqref{Scconditionb}, we find $\displaystyle \left(\frac{1}{\kappa_{\gamma}}\right)^{\prime} \frac{1}{\tau_{\gamma}}= \mp c^{-2} \sqrt{a^2+c^2} \cos(cs)$ which yields 
\begin{equation} \label{Sccondition}
\left[\left(\frac{1}{\kappa_{\gamma}}\right)^{\prime} \frac{1}{\tau_{\gamma}} \right]^{\prime}=\pm c^{-1} \sqrt{a^2+c^2} \sin(cs).
\end{equation}
Now, we find by considering \eqref{Stanh} and \eqref{casea} together
\begin{equation} \label{Scratiocondition}
\frac{\tau_{\gamma}}{\kappa_{\gamma}}= \tanh \left(\int{\tau_{\beta}(s)ds}\right)= \pm \frac{\sqrt{a^2+c^2}}{c} \sin(cs).
\end{equation}
Finally, by considering \eqref{Sccondition} and \eqref{Scratiocondition} in \eqref{Sctype1}, we conclude $\gamma$ 
is a space-like spherical curve of \textit{type1}.

Conversely, assume that $\gamma$ is a spherical curve and its natural mate $\beta$ has constant curvature 
$\kappa_{\beta}=c>0$. Then, the curvature $\kappa_{\gamma}$ and torsion $\tau_{\gamma}$ satisfy
\begin{equation} \label{Srcondition}
\left(\frac{1}{\kappa_{\gamma}}\right)^{2}-\left[{\left(\frac{1}{\kappa_{\gamma}}\right)^{\prime}}{\frac{1}{\tau_{\gamma}}}\right]^{2}=r^2,
\end{equation}
where $r>0$ being constant, (see in \cite{PP1999}). Thus, by substituting \eqref{condition0} 
and \eqref{ScconditionALL} in the last equation, we obtain 
\begin{equation} \nonumber
\tau_{\beta} {\cosh \left(\int{\tau_{\beta}(s)ds}\right)}^{-1}= \pm \cosh \left(\int{\tau_{\beta}(s)ds}\right) \sqrt{a^2+c^2 \cosh^{-2} \left(\int{\tau_{\beta}(s)ds}\right)}
\end{equation}
where $a$ is a constant such that $a\geq c$ and $a=r c^2$. Thus we have
\begin{equation} \nonumber
\frac{c \tau_{\beta} {\cosh \left(\int{\tau_{\beta}(s)ds}\right)}^{-2}}{\sqrt{a^2+c^2 {\cosh \left(\int{\tau_{\beta}(s)ds}\right)}^{-2}}}= \pm c,
\end{equation}
which yields by integrating 
\begin{equation} \nonumber
\arcsin \Big(\frac{c}{a^2+c^2} \tanh \left(\int{\tau_{\beta}(s)ds}\right)\Big)= \pm cs+c_0
\end{equation}
for some constant $c_0$. Therefore, we get
\begin{equation} \nonumber
\tanh \left(\int{\tau_{\beta}(s)ds}\right)= \pm \frac{a^2+c^2}{c} \sin(cs),
\end{equation}
after by applying a suitable translation in $s$. The last equation is equivalent to 
\begin{equation} \nonumber
\int{\tau_{\beta}(s)ds}= \pm \tanh^{-1} \Big( \frac{a^2+c^2}{c} \sin(cs)\Big).
\end{equation}
Finally, by differentiating the above equation yields \eqref{casei}.
\end{proof}

\begin{thm}
Let $\gamma$ be a time-like Frenet curve and its natural mate $\beta$ be of constant curvature $c>0$. 
Then $\gamma$ is a spherical curve if and only if there is a constant $a\geq c$ such that the torsion
 $\tau_{\beta}$ of $\beta$ with respect to a suitable arc-length parametrization $s$ satisfies one of the following:
\begin{itemize}
\item[(i)] If $\left|\kappa_{\gamma}\right|>\left|\tau_{\gamma}\right|$,
\begin{equation} \label{caseii}
\tau_{\beta}= \pm \frac{c^2 \sqrt{a^2+c^2}\sin (cs)}{c^2-(a^2+c^2){\cos^{2}(cs)}}.
\end{equation}
In case, $\beta$ is a space-like curve of \textit{type1}.
\item[(ii)] If $\left|\kappa_{\gamma}\right|<\left|\tau_{\gamma}\right|$,
\begin{equation}
\tau_{\beta}= \pm \frac{c^2 \sqrt{a^2-c^2}\cosh (cs)}{c^2-(a^2+c^2){\sinh(cs)}^2}.
\end{equation}
In case, $\beta$ is a space-like curve of \textit{type2}.
\end{itemize}
\end{thm}

\begin{proof}
Proof of this theorem can be made exactly similar to the proof of the previous theorem.
\end{proof}

\end{document}